\documentclass[final]{dmtcs-episciences}

\usepackage[utf8]{inputenc}
\usepackage{subfigure}

\usepackage{amsmath,amssymb,tikz}
\usetikzlibrary{patterns}

\newtheorem{theorem}{Theorem}[section]
\newtheorem{corollary}[theorem]{Corollary}
\newtheorem{proposition}[theorem]{Proposition}

\newtheorem{conjecture}{Conjecture}

\newtheorem{remark}{Remark}
\newtheorem{definition}[theorem]{Definition}
\numberwithin{equation}{section}

\tikzstyle{mesh}=[pattern=north east lines, pattern color=gray!70, draw=gray]

\def\oeis#1{\cite[#1]{Sloane}}
\def\abs#1{\lvert#1\rvert}
\def\Av{\mathcal{S}}
\def\A{\mathcal{A}}

\DeclareMathOperator\red{red}

\definecolor{red1}{HTML}{B02400}

\title[A positional statistic for $1324$-avoiding permutations]{A positional statistic for $1324$-avoiding permutations}

\author[J.~Gil, O.~Lopez, M.~Weiner]{Juan B. Gil\affiliationmark{1}
  \and Oscar A. Lopez\affiliationmark{2}
  \and Michael D. Weiner\affiliationmark{1}}

\affiliation{
  Penn State Altoona, Altoona, PA, U.S.A.\\
  Penn State Harrisburg, Middletown, PA, U.S.A.}


\begin{document}

\publicationdata{vol. 26:1, Permutation Patterns 2023}{2024}{7}{10.46298/dmtcs.12629}{2023-12-01; 2023-12-01; 2024-05-22}{2024-09-17}

\maketitle

\begin{abstract}~

We consider the class $\Av_n(1324)$ of permutations of size $n$ that avoid the pattern 1324 and examine the subset $\Av_n^{a\prec n}(1324)$ of elements for which $a\prec n\prec [a-1]$, $a\ge 1$. This notation means that, when written in one-line notation, such a permutation must have $a$ to the left of $n$, and the elements of $\{1,\dots,a-1\}$ must all be to the right of $n$. For $n\ge 2$, we establish a connection between the subset of permutations in $\Av_n^{1\prec n}(1324)$ having the 1 adjacent to the $n$ (called {\em primitives}), and the set of 1324-avoiding dominoes with $n-2$ points. For $a\in\{1,2\}$, we introduce constructive algorithms and give formulas for the enumeration of $\Av_n^{a\prec n}(1324)$ by the position of $a$ relative to the position of $n$. For $a\ge 3$, we formulate some conjectures for the corresponding generating functions.
\end{abstract}

%%%%%%%%%%%%%%%%%%%%%%%%%%%%%%%%%%%%%%%%%%%%
\section{Introduction}
\label{sec:intro}

Finding a closed formula for the enumeration of the class $\Av_n(1324)$ remains an open problem that we do NOT solve in this paper. However, in our attempt to understand what makes this class so difficult to handle, we came across a simple but interesting type of statistics: distance between the smallest and largest element of a permutation. We found that permutations of size $n$ having the $1$ adjacent to the $n$ are manageable and can be used to enumerate certain related subsets of $\Av_n(1324)$.

The goal of this paper is to present our findings and formulate a conjecture for the enumeration of other subsets of $\Av_n(1324)$ according to similar positional statistics.

For $a,k\ge 1$, let $\Av_{n,k}^{a\prec n}(1324)$ be the set of permutations $\sigma\in \Av_n(1324)$ such that:
\begin{itemize}
\item $\sigma^{-1}(n)-\sigma^{-1}(a)=k$,
\item $\sigma^{-1}(b)-\sigma^{-1}(n)>0$ for every $b\in\{1,\dots,a-1\}$.
\end{itemize}
Let $\Av_{n}^{a\prec n}(1324)=\bigcup\limits_{k\ge 1} \Av_{n,k}^{a\prec n}(1324)$ and observe that, in one-line notation, permutations in the set $\Av_{n}^{a\prec n}(1324)$ have the entry $a$ to the left of $n$ and all the elements of $\{1,\dots,a-1\}$ to the right of $n$. We also let 
\[ T_{a, k}(x)=\sum\limits_{n=k+1}^{\infty}\abs{\Av_{n,k}^{a\prec n}(1324)}x^n \;\text{ and }\; g_a(x,t) 
  = \sum\limits_{k=1}^{\infty}t^kT_{a,k}(x). \]

In Section~\ref{sec:1<n}, we discuss $\Av_{n,k}^{1\prec n}(1324)$ and give formulas for $T_{1,k}(x)$ and $g_1(x,t)$. First, we give a bijection between $\Av_{n,1}^{1\prec n}(1324)$ and the set of 1324-avoiding dominoes with $n-2$ points, which is known to be counted by the OEIS sequence \oeis{A000139}. We then introduce a product $\sigma_1\odot\sigma_2$ on $\Av_{n,1}^{1\prec n}(1324)$ that we use to construct and enumerate the elements of $\Av_{n,k}^{1\prec n}(1324)$. In Section~\ref{sec:2<n}, we examine the set $\Av_{n,k}^{2\prec n}(1324)$ and give explicit formulas for $T_{2,k}(x)$ and $g_2(x,t)$. Finally, in the last section of the paper, we conjecture a formula for $T_{a, k}(x)$ for $3\le a\le k$.

The significance of $g_a(x,t)$ lies in the fact that if $G(x) = \sum\limits_{n=1}^{\infty} |\Av_n(1324)| x^n$, then
\begin{equation*} 
 G(x) = \frac{1}{1-x}\bigg(x+\sum_{a=1}^{\infty}g_a(x,1)\bigg).
\end{equation*}
Note that for $n\ge 2$, permutations in $\Av_n(1324)$ that start with $n$ are counted by the function $xG(x)$. Thus $G(x) = x + xG(x) + \sum\limits_{a=1}^{\infty}g_a(x,1)$.

Again, we are still far from giving a formula for $G(x)$, but the breakdown using the sets $\Av_{n,k}^{a\prec n}(1324)$ provides a different viewpoint that we believe is worth pursuing further.

%%%%%%%%%%%%%%%%%%%%%%%%%%%%%%%%%%%%%%%%%%%%
\section{Enumeration of $\Av_n^{1\prec n}(1324)$}
\label{sec:1<n}

We start by establishing a bijection between $\Av_{n,1}^{1\prec n}(1324)$ and the set of 1324-avoiding dominoes. As studied by D.~Bevan, R.~Brignall, A.~Elvey Price, and J.~Pantone \cite{BBPP20}, a 1324-avoiding vertical {\em domino} is a two-cell gridded permutation in $\textup{Grid}^{\#}\binom{\textup{Av}(213)}{\textup{Av}(132)}$ whose underlying permutation avoids 1324. These dominoes are counted by the sequence $1, 2, 6, 22, 91, 408, 1938, 9614,\dots$, cf.~\oeis{A000139}.

For example, the six distinct $1324$-avoiding dominoes with two points are

\medskip
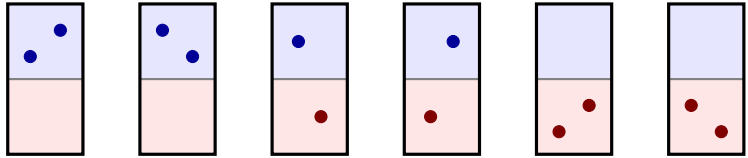
\begin{center}
\def\rad{0.08}
\begin{tikzpicture}
\begin{scope}
\fill[blue!10] (0,1) rectangle (1,2);
\fill[red!10] (0,0) rectangle (1,1);
\draw[thick, gray] (0,1) -- (1,1);
\draw[very thick] (0,0) rectangle (1,2);
\draw[fill,blue!60!black] (0.3,1.3) circle (\rad); 
\draw[fill,blue!60!black] (0.7,1.65) circle (\rad); 
\end{scope}
\begin{scope}[xshift=50]
\fill[blue!10] (0,1) rectangle (1,2);
\fill[red!10] (0,0) rectangle (1,1);
\draw[thick, gray] (0,1) -- (1,1);
\draw[very thick] (0,0) rectangle (1,2);
\draw[fill,blue!60!black] (0.7,1.3) circle (\rad); 
\draw[fill,blue!60!black] (0.3,1.65) circle (\rad); 
\end{scope}
\begin{scope}[xshift=100]
\fill[blue!10] (0,1) rectangle (1,2);
\fill[red!10] (0,0) rectangle (1,1);
\draw[thick, gray] (0,1) -- (1,1);
\draw[very thick] (0,0) rectangle (1,2);
\draw[fill,blue!60!black] (0.35,1.5) circle (\rad); 
\draw[fill,red!50!black] (0.65,0.5) circle (\rad); 
\end{scope}
\begin{scope}[xshift=150]
\fill[blue!10] (0,1) rectangle (1,2);
\fill[red!10] (0,0) rectangle (1,1);
\draw[thick, gray] (0,1) -- (1,1);
\draw[very thick] (0,0) rectangle (1,2);
\draw[fill,red!50!black] (0.35,0.5) circle (\rad); 
\draw[fill,blue!60!black] (0.65,1.5) circle (\rad); 
\end{scope}
\begin{scope}[xshift=200]
\fill[blue!10] (0,1) rectangle (1,2);
\fill[red!10] (0,0) rectangle (1,1);
\draw[thick, gray] (0,1) -- (1,1);
\draw[very thick] (0,0) rectangle (1,2);
\draw[fill,red!50!black] (0.3,0.3) circle (\rad); 
\draw[fill,red!50!black] (0.7,0.65) circle (\rad); 
\end{scope}
\begin{scope}[xshift=250]
\fill[blue!10] (0,1) rectangle (1,2);
\fill[red!10] (0,0) rectangle (1,1);
\draw[thick, gray] (0,1) -- (1,1);
\draw[very thick] (0,0) rectangle (1,2);
\draw[fill,red!50!black] (0.7,0.3) circle (\rad); 
\draw[fill,red!50!black] (0.3,0.65) circle (\rad); 
\end{scope}
\end{tikzpicture}
\end{center}

\begin{proposition} \label{prop:DominoesBijection}
For $n\ge 2$, there is a one-to-one correspondence between $\Av_{n,1}^{1\prec n}(1324)$ and the set of $1324$-avoiding dominoes with $n-2$ points.
\end{proposition}
\begin{proof}
The bijection relies on the inverse map. Every $\sigma\in \Av_{n,1}^{1\prec n}(1324)$ is of the form 
\[ \sigma = \sigma_L \,1 n\, \sigma_R, \]
where $\sigma_L$ and $\sigma_R$ are words (possibly empty) such that $|\sigma_L|+|\sigma_R|=n-2$, and their reduced permutations\footnote{The permutation $\red(\sigma)$ is obtained by replacing the $i$th smallest letter of $\sigma$ by $i$, for $i=1,\dots,|\sigma|$.} $\red(\sigma_L)$ and $\red(\sigma_R)$ avoid 132 and 213, respectively. Clearly, since $1324$ is an involution, $\sigma^{-1}$ also avoids 1324. Moreover, if $i=\sigma^{-1}(1)$, then $\sigma^{-1}$ is of the form
\begin{center}
\begin{tikzpicture}[scale=0.7]
\def\a{1.35}
\def\b{1.65}
\node[left=2] at (0,\a) {\small $i$};
\node[right=2] at (3,\b) {\small $i+1$};
\draw[fill,blue!10] (0,\b) rectangle (3,3);
\draw[fill,red!10] (0,0) rectangle (3,\a);
\draw[mesh] (0,\a) rectangle (3,\b);
\draw[gray] (0,0) rectangle (3,3);
\draw[gray] (0,\a) -- (3,\a);
\draw[gray] (0,\b) -- (3,\b);
\draw[fill] (0,\a) circle (0.1);
\draw[fill] (3,\b) circle (0.1);
\node[above=0] at (1.5,2) {\small $\sigma_T$};
\node[below=0] at (1.5,1) {\small $\sigma_B$};
\end{tikzpicture}
\end{center}
where $\red(\sigma_T)=\red(\sigma_R)^{-1}$ avoids 213 and $\red(\sigma_B)=\red(\sigma_L)^{-1}$ avoids 132. Finally, merging the lines through $i$ and $i+1$ as a separator, we get a 1324-avoiding domino.
\end{proof}

As a consequence (\cite[Theorem~2]{BBPP20} with $n$ replaced by $n-2$), we have
\[ \abs{\Av_{n,1}^{1\prec n}(1324)} = \frac{2\,(3n-3)!}{(2n-1)! n!} \;\text{ for } n\ge 2. \]

The elements of $\Av_{n,1}^{1\prec n}(1324)$ serve as primitives to factor and enumerate the elements of $\Av_{n,k}^{1\prec n}(1324)$ for every $k\ge 2$.

\medskip
We start with a definition.
\begin{definition}
A permutation $\sigma\in \Av_{n,1}^{1\prec n}(1324)$ will be called a {\em primitive}. For $n\ge 2$, such a permutation 
must be of the form $\sigma=\pi\,1 n\,\tau$ with $\red(\pi)\in \Av_k(132)$, $\red(\tau)\in \Av_\ell(213)$, and $k+\ell=n-2$. Given a primitive $\sigma_1=\pi_1\,1 m\,\tau_1\in \Av_m(1324)$ and a permutation $\sigma_2\in \Av_\ell(1324)$ of the form $\sigma_2=\pi_2\,1\,\theta_2\,\ell\,\tau_2$ with $|\theta_2|\ge 0$, we define the product
\begin{equation}\label{eq:dotProduct}
 \sigma_1\odot \sigma_2 = \widehat{\pi}_2\pi_1\,1m\,\widehat{\theta}_2\,n\,\widehat{\tau}_2\tau_1 \in \Av_n,
\end{equation}
where $n=\ell+m-1$, and $\widehat{\pi}_2$, $\widehat{\theta}_2$, and $\widehat{\tau}_2$ are obtained from $\pi_2$, $\theta_2$, and $\tau_2$, by increasing all of their entries by $m-1$ (see Figure~\ref{fig:dotProduct}). 
%If $\varepsilon$ is the empty permutation, we let $\sigma\odot\varepsilon=\sigma$.
\end{definition}
\begin{figure}[ht]
\centering
\begin{tikzpicture}[scale=1.2]
\tikzstyle{pBox}=[fill,opacity=0.4]
\draw[mesh] (2,0.25) rectangle (2.25,3.25);
\draw[gray!60] (0,0.25) rectangle (5.5,3.25);
\draw[gray!60, thick] (0,1.75)--(5.5,1.75); %horizontal line through m
\foreach \x in {2,2.25,3.5}{\draw[gray!60] (\x,0.25)--(\x,3.25);}
\draw[pBox,brown] (0,2) rectangle (0.75,3); %\pi_2
\draw[pBox,olive] (1,0.5) rectangle (1.75,1.5); %\pi_1
\draw[pBox,brown] (2.5,2) rectangle (3.25,3); %\theta_2
\draw[pBox,brown] (3.75,2) rectangle (4.5,3); %\tau_2
\draw[pBox,olive] (4.75,0.5) rectangle (5.5,1.5); %\tau_1
\node at (0.38,2.5) {$\hat{\pi}_2$};
\node at (1.38,1) {$\pi_1$};
\node at (2.88,2.5) {$\hat{\theta}_2$};
\node at (4.13,2.5) {$\hat{\tau}_2$};
\node at (5.13,1) {$\tau_1$};
\draw[fill,olive] (2,0.25) circle(0.06);
\draw[fill] (2.25,1.75) circle(0.06);
\draw[fill,brown] (3.5,3.25) circle(0.06);
\node[below] at (2,0.25) {\small $1$};
\node[above] at (2.25,1.75) {\small $m$};
\node[above] at (3.5,3.25) {\small $n$};
\end{tikzpicture}
\caption{Visualization of $\sigma_1\odot\sigma_2$.}
\label{fig:dotProduct}
\end{figure}
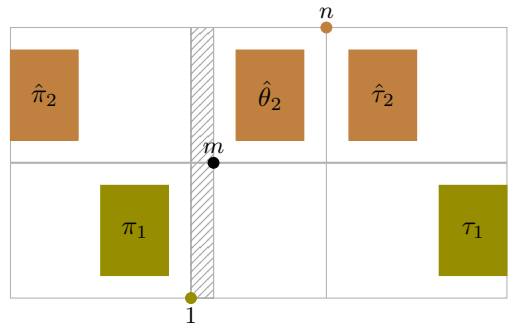

For instance, $2143\odot 41253 = 7\,21\, 4\,586\,3$. Also, 
\[ 213\odot 3142=5\, 21\, 364 \;\text{ and }\; 3142\odot 213= 5\, 31\, 46\, 2. \] 
In particular, $\odot$ is not commutative. This is true even if $\sigma_1$ and $\sigma_2$ are both primitives. For example, $213\odot 12 = 2134$ but $12\odot 213 = 3124$.

\begin{proposition}
If $\sigma_1 \in \Av_{m,1}^{1\prec m}(1324)$ and $\sigma_2 \in \Av_{\ell,k}^{1\prec \ell}(1324)$, then
\[ \sigma_1\odot \sigma_2 \in \Av_{n,k+1}^{1\prec n}(1324) \text{ with } n=\ell+m-1.\]
\end{proposition}
\begin{proof}
The permutations $\sigma_1$ and $\sigma_2$ must be of the form 
\[ \sigma_1=\pi_1\,1m\,\tau_1 \;\text{ and }\; \sigma_2=\pi_2\,1\theta_2\,\ell\,\tau_2, \]
where $\red(\pi_1)$ and $\red(\pi_2)$ both avoid 132, $\red(\tau_1)$ and $\red(\tau_2)$ both avoid 213, $\red(\theta_2)$ avoids $21$, and $|\theta_2|=k-1$. By \eqref{eq:dotProduct}, we have $\sigma_1\odot \sigma_2\in\Av_{n,k+1}^{1\prec n}$, so we only need to show that $\sigma_1\odot \sigma_2$ avoids 1324.

The graph of $\sigma_1\odot \sigma_2$ (see Figure~\ref{fig:dotProduct}) makes it clear that, since both $\sigma_1$ and $\sigma_2$ avoid 1324, the parts of the permutation below and above the horizontal line $y=m$ cannot have a 1324 pattern. Thus, if there is a 1324 pattern in $\sigma_1\odot \sigma_2$, the 1 will have to be in $\pi_1\,1$, and since $\pi_1$ avoids 132, the 3, 2, and 4 of the pattern will have to be above the line $y=m$. But this is not possible since $\theta_2\,\ell\,\tau_2$ avoids 213.
\end{proof}

\begin{proposition}
Every non-primitive $\sigma\in \Av_n^{1\prec n}(1324)$ admits a unique decomposition 
\[ \sigma=\sigma_1\odot \sigma_2, \] 
where $\sigma_1$ is a primitive in $\Av_{m,1}^{1\prec m}(1324)$ and $\sigma_2\in\Av_\ell^{1\prec\ell}(1324)$ with $\ell=n-m+1$.
\end{proposition}
\begin{proof}
If $\sigma$ is not primitive, then it must be of the form $\sigma=\pi\,1\theta n\,\tau$, where $\red(\pi)$ avoids 132, $\red(\tau)$ avoids 213, and $\theta$ is increasing (i.e., avoids 21). Let $m=\min(\theta)$. If there were indices $i<j$ such that $\pi(i)<m<\pi(j)$, then the sequence $(\pi(i),\pi(j),m,n)$ would make a forbidden 1324 pattern:
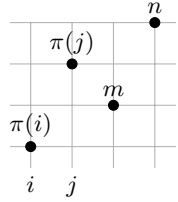
\begin{center}
\begin{tikzpicture}[scale=0.55]
\node[below] at (1,0.5) {\small $i$};
\node[below] at (2,0.5) {\small $j$};
\clip (0.5,0.5) rectangle (4.5,4.5);
\draw[gray!60] (0,0) grid (5,4);
\foreach \x/\y in {1/1,2/3,3/2,4/4}{\draw[fill] (\x,\y) circle (0.12);}
\node[above] at (1,1) {\small $\pi(i)$};
\node[above] at (2,3) {\small $\pi(j)$};
\node[above] at (3,2) {\small $m$};
\node[above] at (4,4) {\small $n$};
\end{tikzpicture}
\end{center}
Thus, if $\pi$ is not empty, it must be of the form $\pi=\pi_2\pi_1$ where $\pi_2$ consists of the values of $\pi$ that are larger than $m$, and $\pi_2$ contains the values of $\pi$ smaller than $m$, if any. Similarly, if $\tau(i)<m<\tau(j)$ for some $i<j$, then the sequence $(1,m,\tau(i),\tau(j))$ would make a forbidden 1324 pattern. Thus, if $\tau$ is not empty, it must be of the form $\tau=\tau_2\tau_1$ where the values of $\tau_2$ (if not empty) are larger than $m$ and the values of $\tau_1$ are smaller than $m$.

Now, if $\theta'$ denotes $\theta$ without the entry $m$, then the graph of $\sigma$ takes the form
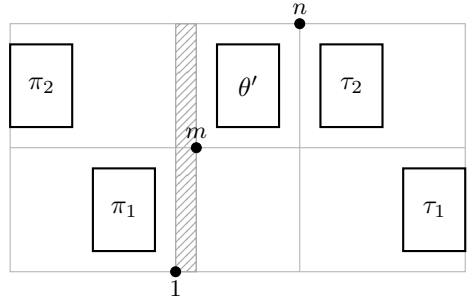
\begin{center}
\begin{tikzpicture}[scale=1.1]
\draw[mesh] (2,0.25) rectangle (2.25,3.25);
\draw[gray!60] (0,0.25) rectangle (5.5,3.25);
\draw[gray!60] (0,1.75)--(5.5,1.75); %horizontal line through m
\foreach \x in {2,2.25,3.5}{\draw[gray!60] (\x,0.25)--(\x,3.25);}
\draw[thick] (0,2) rectangle (0.75,3); %\pi_2
\draw[thick] (1,0.5) rectangle (1.75,1.5); %\pi_1
\draw[thick] (2.5,2) rectangle (3.25,3); %\theta_2
\draw[thick] (3.75,2) rectangle (4.5,3); %\tau_2
\draw[thick] (4.75,0.5) rectangle (5.5,1.5); %\tau_1
\node at (0.38,2.5) {$\pi_2$};
\node at (1.38,1) {$\pi_1$};
\node at (2.88,2.5) {$\theta'$};
\node at (4.13,2.5) {$\tau_2$};
\node at (5.13,1) {$\tau_1$};
\draw[fill] (2,0.25) circle(0.06);
\draw[fill] (2.25,1.75) circle(0.06);
\draw[fill] (3.5,3.25) circle(0.06);
\node[below] at (2,0.25) {\small $1$};
\node[above] at (2.25,1.75) {\small $m$};
\node[above] at (3.5,3.25) {\small $n$};
\end{tikzpicture}
\end{center}
where any box could be empty. If we let $\sigma_1=\pi_1\,1m\,\tau_1$ and $\sigma_2=\red(\pi_2\,m \theta' n\,\tau_2)$, then $\sigma_1$ is primitive, $\sigma_2\in\Av_\ell^{1\prec\ell}(1324)$ with $\ell=n-m+1$, and it is easy to see that $\sigma=\sigma_1\odot\sigma_2$.

Now suppose $\sigma=\rho_1\odot\rho_2$ with a primitive $\rho_1$ of the form $\rho_1=\alpha_1\, 1m'\,\beta_1$. The definition of $\rho_1\odot\rho_2$ implies $m'=m$ and all the values of $\sigma$ less than or equal to $m$ must coincide with the values of $\rho_1$. Hence $\rho_1=\sigma_1$ and the factorization is unique.
\end{proof}

\begin{corollary} \label{cor:1<n}
Every permutation in $\Av_{n,k}^{1\prec n}(1324)$ can be uniquely decomposed as a product of $k$ primitive permutations. 
\end{corollary}
\begin{table}[ht]
\begin{align*}
(k=3)\qquad & 1234 = 12\odot 12\odot 12 \\[3pt]
(k=2)\qquad & 1243 = 12\odot 132 \\
 & 1342 = 132\odot 12 \\
 & 2134 = 213\odot 12 \\
 & 3124 = 12\odot 213\\[3pt]
(k=1)\qquad & 1423, 1432, 2143, 3142, 2314, 3214 \; \text{ (primitives)}
\end{align*}
\caption{Elements of $\Av_{4,k}^{1\prec 4}(1324)$ and their primitive decomposition.} 
\label{tab:composite_n=4}
\end{table}

In general, every permutation $\sigma\in\Av_{n,k}^{1\prec n}(1324)$ with $k$ primitive components is of the form depicted in Figure~\ref{fig:k-dotProduct}, where $\Theta$ is increasing (possibly empty), and each primitive component can be read from the horizontal regions determined by the values of $\sigma$ to the right of $1$ and to the left of $n$.

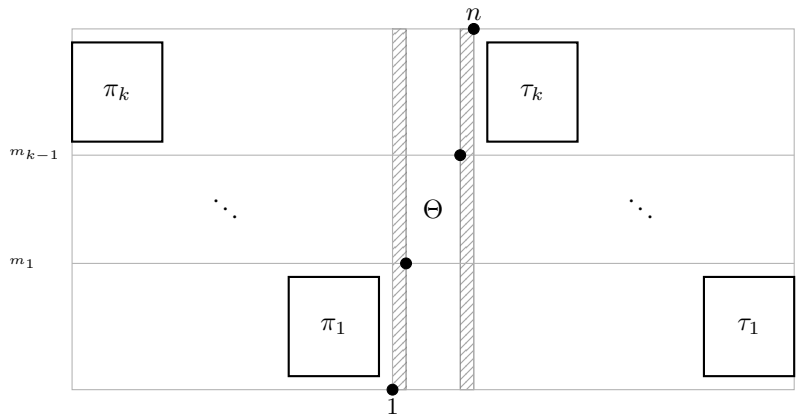
\begin{figure}[ht]
\centering
\begin{tikzpicture}[scale=1.2]
\draw[mesh] (3.55,0) rectangle (3.7,4);
\draw[mesh] (4.3,0) rectangle (4.45,4);
\draw[gray!60] (0,0) rectangle (8,4);
\foreach \y in {1.4,2.6}{\draw[gray!60] (0,\y)--(8,\y);}
\foreach \x in {3.55,4.45}{\draw[gray!60] (\x,0)--(\x,4);}
\draw[thick] (0,2.75) rectangle (1,3.85); %\pi_k
\draw[thick] (2.4,0.15) rectangle (3.4,1.25); %\pi_1
\draw[thick] (4.6,2.75) rectangle (5.6,3.85); %\tau_k
\draw[thick] (7,0.15) rectangle (8,1.25); %\tau_1
\node at (0.5,3.3) {$\pi_k$};
\node at (1.7,2.1) {$\ddots$};
\node at (2.9,0.7) {$\pi_1$};
\node at (4,2) {$\Theta$};
\node at (5.1,3.3) {$\tau_k$};
\node at (6.3,2.1) {$\ddots$};
\node at (7.5,0.7) {$\tau_1$};
\draw[fill] (3.55,0) circle(0.06);
\draw[fill] (3.7,1.4) circle(0.06);
\draw[fill] (4.3,2.6) circle(0.06);
\draw[fill] (4.45,4) circle(0.06);
\node[below] at (3.55,0) {\small $1$};
\node[right] at (-0.8,1.4) {\tiny $m_1$};
\node[right] at (-0.8,2.6) {\tiny $m_{k-1}$};
\node[above] at (4.45,4) {$n$};
\end{tikzpicture}
\caption{$\sigma = \sigma_1\odot\cdots\odot\sigma_k$.}
\label{fig:k-dotProduct}
\end{figure}

\medskip
As mentioned at the beginning of this section, the elements of $\Av_{n,1}^{1\prec n}(1324)$ are counted by the OEIS sequence \oeis{A000139}. Let $f(x)$ denote the generating function for A000139 without the constant term, that is,
\begin{equation}\label{eq:A139(x)} 
  f(x) = x + 2x^2 + 6x^3 + 22x^4 + 91x^5 + 408x^6 + 1938x^7 + 9614x^8 +\cdots,
\end{equation}
and recall the notation (given in the introduction)
\[  T_{1,k}(x)=\sum\limits_{n=k+1}^{\infty}\,\abs{\Av_{n,k}^{1\prec n}(1324)}x^n \;\text{ and }\; g_1(x,t) = \sum\limits_{k=1}^{\infty}t^kT_{1,k}(x). \]
By Proposition~\ref{prop:DominoesBijection}, we have $T_{1,1}(x) = xf(x)$. More generally:

\begin{theorem} \label{thm:1<n}
Let $a_{n,k} = \abs{\Av_{n,k}^{1\prec n}(1324)}$. For $n\ge 3$ and $2\le k\le n-1$, we have
\[ a_{n,k} = \sum_{m=2}^{n-k+1} a_{m,1}\cdot a_{n-m+1,k-1}. \]
Consequently, $T_{1,k}(x) = xf(x)^k$ and thus $g_1(x,t) = \dfrac{xtf(x)}{1-tf(x)}$.
\end{theorem}
\begin{proof}
By the above propositions, the permutations in $\Av_{n,k}^{1\prec n}(1324)$ can be built by taking all possible products $\sigma_1\odot\sigma_2$, where $\sigma_1$ is a primitive of size $m$ for some $m\ge 2$, and $\sigma_2$ is a 1324-avoiding permutation of size $n-m+1$ with $\sigma_2^{-1}(n-m+1)-\sigma_2^{-1}(1)=k-1$. There are $a_{m,1}$ permutations of the first kind, and $a_{n-m+1,k-1}$ of the latter. This leads to the claimed formula for $a_{n,k}$. Observe that $a_{n-m+1,k-1}$ only makes sense whenever $n-m+1\ge k$, so we need $m\le n-k+1$. 

The formula for $T_{1,k}(x)$ can be shown by induction in $k$. As mentioned above, $T_{1,1}(x) = xf(x)$. Suppose $T_{1,k-1}(x)=xf(x)^{k-1}$. Then,
\begin{align*}
 xT_{1,k} = \sum_{n=k+1}^\infty a_{n,k} x^{n+1} 
 &= \sum_{n=k+1}^\infty \bigg(\sum_{m=2}^{n-k+1} a_{m,1}\cdot a_{n-m+1,k-1}\bigg) x^{n+1} \\
 &= \sum_{n=k+2}^\infty \bigg(\sum_{m=2}^{n-k} a_{m,1}\cdot a_{n-m,k-1}\bigg) x^{n} \\
 &= \bigg(\sum_{m=2}^\infty a_{m,1} x^{m}\bigg) \bigg(\sum_{\ell=k}^\infty a_{\ell,k-1} x^{\ell}\bigg).
\end{align*}
In other words, $xT_{1,k} = T_{1,1}(x)T_{1,k-1}(x)$ and therefore $T_{1,k}(x) = xf(x)^k$, as claimed. 
\end{proof}

%%%%%%%%%%%%%%%%%%%%%%%%%%%%%%%%%%%%%%%%%%%%
\section{Enumeration of $\Av_n^{2 \prec n}(1324)$}
\label{sec:2<n}

In this section, we focus on the set $\Av_{n,k}^{2\prec n}(1324)$ and give formulas for the functions 
\[  T_{2,k}(x)=\sum\limits_{n=k+1}^{\infty}\,\abs{\Av_{n,k}^{2\prec n}(1324)}x^n \;\text{ and }\; g_2(x,t) = \sum\limits_{k=1}^{\infty}t^kT_{2,k}(x) \]
in terms of the functions $f(x)$ and $g_1(x,t)$ from the previous section. Recall that $\Av_{n,k}^{2\prec n}(1324)$ is the set of permutations $\sigma\in \Av_n(1324)$ having (in one-line notation) the $2$ to the left of $n$ at distance $k$, and the 1 to the right of $n$.
Note that if the 1 is removed from such a permutation, the reduced permutation is an element of $\Av_{n-1,k}^{1\prec n-1}(1324)$. This basic observation is used as guideline for most of our combinatorial arguments.

As before, we let $a_{n,k} = \abs{\Av_{n,k}^{1\prec n}(1324)}$, so $g_1(x,t)=\sum\limits_{k=1}^\infty\sum\limits_{n=k+1}^{\infty} a_{n,k}\, t^k x^n $.

\begin{theorem} \label{thm:2<n}
For $n\ge 3$ and $2\le k\le n-1$, we have
\begin{equation} \label{eq:2<nTriangle}
 \abs{\Av_{n,k}^{2\prec n}(1324)} = \tfrac{n-k}{2}a_{n-1,k}.
\end{equation}
Moreover, the corresponding generating function satisfies
\[ g_2(x,t) = \frac12\Big(x^2\frac{\partial g_1}{\partial x}(x,t)-g_1(x,t)^2\Big). \]
\end{theorem}
\begin{proof}
Let $\A(n,k) = \Av_{n,k}^{1\prec n}(1324)$. For every $\sigma\in\A(n-1,k)$, we let 
\[ i=\sigma^{-1}(1)-1 \;\text{ and }\; j=n-1-\sigma^{-1}(n-1). \] 
Thus, there are $i$ entries to the left of 1, $j$ entries to the right of $n-1$, and $i+j+k+1=n-1$. 

By inserting 1 at any of the the $j+1$ positions to the right of $n-1$, the permutation $\sigma$ gives rise to $j+1$ permutations in $\Av_{n,k}^{2\prec n}(1324)$, and with a similar process, the reverse complement $\sigma^{rc}$ leads to $i+1$ such permutations. In other words, the pair $(\sigma,\sigma^{rc})$ produces a total of $i+1+j+1=n-k$ permutations in $\Av_{n,k}^{2\prec n}(1324)$, but so does the pair $(\sigma^{rc}, \sigma)$. Therefore, as we go over all permutations $\sigma \in \A(n-1,k)$, the above process generates all of the elements of $\Av_{n,k}^{2\prec n}(1324)$ twice. Thus there are $\frac{n-k}{2} a_{n-1,k}$ permutations in $\Av_{n,k}^{2\prec n}(1324)$.

Now, given that $g_1(x,t)=\sum\limits_{k=1}^\infty\sum\limits_{n=k+1}^{\infty} a_{n,k}\, t^k x^n$, formula \eqref{eq:2<nTriangle} implies 
\begin{equation} \label{eq:g2g1}
\begin{aligned} 
g_2(x,t) &= \frac12\bigg(x\frac{\partial}{\partial x}\Big(xg_1(x,t)\Big)-tx\frac{\partial}{\partial t}g_1(x,t)\bigg) \\ 
  &= \frac12\Big(x^2\frac{\partial}{\partial x}g_1(x,t) + xg_1(x,t)-tx\frac{\partial}{\partial t}g_1(x,t)\Big).
\end{aligned}
\end{equation}

Finally, since $g_1(x,t)=\frac{xtf(x)}{1-tf(x)}$ by Theorem~\ref{thm:1<n}, we get $\frac{\partial}{\partial t}g_1(x,t)=\frac{xf(x)}{\left(1-tf(x)\right)^2}$, and it follows that $xg_1(x,t)-tx\frac{\partial}{\partial t}g_1(x,t) = -g_1(x,t)^2$. 
\end{proof}

\medskip
There is a formula for $T_{2, k}(x)$ that seems more suitable for generalizations. We start by letting $T_{2,0}(x)=x^2$ (to account for the permutation $21$, where the 2 is at distance zero from the maximal element). Moreover, since $T_{1,1}(x)=\sum\limits_{n=2}^{\infty}a_{n,1}x^n$ and $\abs{\Av_{n,1}^{2\prec n}(1324)} = \frac{n-1}{2} a_{n-1,1}$ by \eqref{eq:2<nTriangle}, we have 
\begin{equation*}
  T_{2, 1}(x) =\sum\limits_{n=3}^{\infty} \frac{n-1}{2} a_{n-1,1}\,x^n = \frac{1}{2} x^2\frac{d}{dx}T_{1,1}(x).
\end{equation*}
Note that by Theorem~\ref{thm:1<n}, this identity can be written as $T_{2, 1}(x)=\frac12 x^2 \frac{d}{dx}(x f(x))$. 

\begin{theorem} \label{thm:T_2k}
For every $k\ge 2$, we have
\[ T_{2,k}(x) = f(x)^kT_{2,0}(x) + kf(x)^{k-1}\big(T_{2,1}(x) - f(x)T_{2,0}(x)\big). \]
\end{theorem}
\begin{proof}
Every permutation $\sigma\in\Av_{n,k}^{2\prec n}(1324)$ that ends with $1$ is of the form $\sigma=\sigma_1\ominus 1$, where $\sigma_1\in \Av_{n-1,k}^{1\prec n-1}(1324)$. Thus, this subset (of permutations in $\Av_{n,k}^{2\prec n}(1324)$ ending with 1) is counted by the function $xT_{1,k}(x)$, which by Theorem~\ref{thm:1<n} equals $x\big(xf(x)^k\big) = f(x)^kT_{2,0}(x)$.

Let $\A_{2,1}(n)$ be the set of permutations in $\Av_{n,1}^{2\prec n}(1324)$ not ending with $1$. This set is counted by the generating function $T_{2,1}(x) - f(x)T_{2,0}(x)$ (all permutations in $\Av_{n,1}^{2\prec n}(1324)$ minus those ending with 1). 

We will show that every element $\sigma\in\Av_{n,k}^{2\prec n}(1324)$ with $\sigma(n)>1$ can be built from (and uniquely corresponds to) a $k$-tuple of permutations $(\sigma_1,\dots,\sigma_k)$ such that:
\begin{itemize}
\item $\sigma_\ell\in \A_{2,1}(m_\ell)$ for some $\ell\in\{1,\dots,k\}$, $m_\ell>3$,
\item $\sigma_j\in \Av_{m_j,1}^{1\prec m_j}(1324)$ for every $j\not=\ell$, $m_j\ge 2$,
\item $\abs{\sigma_1}+\cdots+ \abs{\sigma_k}= n+k-1$.
\end{itemize}
The construction goes as follows. Given such a $k$-tuple, identify $\sigma_\ell$, mark the entry adjacent to the right of $1$ (always possible since $\sigma_\ell$ does not end with $1$), and let $\sigma'_\ell$ be the reduced permutation obtained from $\sigma_\ell$ by removing the $1$. For example, given $(12,2413,132)$, we have $\sigma_1=12$, $\sigma_2=241\underline{3}$, $\sigma_3=132$, and so $\sigma'_2=13\underline{2}$.

Let $\sigma' = \sigma_1\odot \cdots \odot \sigma'_\ell \odot \cdots \odot \sigma_k$, and let $i$ be the position of the transformed (via the dot product) marked entry. Note that $\sigma'$ belongs to $\Av_{n-1,k}^{1\prec n-1}(1324)$, and the position of $n-1$ in $\sigma'$ is less than $i$. For the above example, we get
\[ \sigma' = 12\odot 13\underline{2} \odot 132 = 12 \odot 1354\underline{2} = 12465\underline{3}\]

Next, we let $\sigma$ be the permutation obtained by inserting $1$ into $\sigma'$ at position $i$. By our construction, $\sigma\in \Av_{n,k}^{2\prec n}(1324)$ and $\sigma(n)>1$. For example, the tuple $(12,2413,132)$ yields the marked permutation $\sigma'=12465\underline{3}$, and so $\sigma=2357614$. The permutations in $\Av_{7,3}^{2\prec 7}(1324)$ not ending with $1$ are all listed in Table~\ref{tab:(7,3)} together with their corresponding $3$-tuples.

Conversely, given $\sigma\in\Av_{n,k}^{2\prec n}(1324)$ not ending with $1$, mark the entry adjacent to the right of $1$, remove the 1, and decompose the reduced permutation $\sigma'\in\Av_{n-1,k}^{1\prec n-1}(1324)$ as a product of $k$ primitives. Use the factors to make a $k$-tuple. Identify the component with the marked entry, call it $\sigma'_\ell$, and insert 1 at the position of the mark to obtain an element $\sigma_\ell$ of $\A_{2,1}$. The resulting  $k$-tuple is of the form described above.

Finally, since the generating function for $\abs{\A_{2,1}(n)}$ is $T_{2,1}(x) - f(x)T_{2,0}(x)$, the set of $k$-tuples described above, and thus the set of permutations in $\Av_{n,1}^{2\prec n}(1324)$ not ending with $1$, are enumerated by the generating function $kf(x)^{k-1}\big(T_{2,1}(x) - f(x)T_{2,0}(x)\big)$.
\end{proof}

\begin{remark}
It is worth noting that the above formula for $T_{2,k}(x)$ can also be directly obtained from \eqref{eq:g2g1} together with the fact that $T_{1,k}(x)=xf(x)^k$ and $T_{2, 1}(x)=\frac12 x^2 \frac{d}{dx}(x f(x))$.
\end{remark}
\begin{table}[ht]
\small
\begin{tabular}{rl}
(25134, 12, 12) \;$\leadsto$ & 2567134 \\
(12, 25134, 12) \;$\leadsto$ & 2367145 \\
(12, 12, 25134) \;$\leadsto$ & 2347156 \\
(25143, 12, 12) \;$\leadsto$ & 2567143 \\
(12, 25143, 12) \;$\leadsto$ & 2367154 \\
(12, 12, 25143) \;$\leadsto$ & 2347165 \\
(25413, 12, 12) \;$\leadsto$ & 2567413 \\
(12, 25413, 12) \;$\leadsto$ & 2367514 \\
(12, 12, 25413) \;$\leadsto$ & 2347615 \\
(25314, 12, 12) \;$\leadsto$ & 2567314
\end{tabular}
\quad
\begin{tabular}{rl}
(12, 25314, 12) \;$\leadsto$ & 2367415 \\
(12, 12, 25314) \;$\leadsto$ & 2347516 \\
(32514, 12, 12) \;$\leadsto$ & 3256714 \\
(12, 32514, 12) \;$\leadsto$ & 4236715 \\
(12, 12, 32514) \;$\leadsto$ & 5234716 \\
(42513, 12, 12) \;$\leadsto$ & 4256713 \\
(12, 42513, 12) \;$\leadsto$ & 5236714 \\
(12, 12, 42513) \;$\leadsto$ & 6234715 \\
(2413, 132, 12) \;$\leadsto$ & 2467513 \\
(2413, 12, 132) \;$\leadsto$ & 2457613
\end{tabular}
\quad
\begin{tabular}{rl}
(132, 2413, 12) \;$\leadsto$ & 2467153 \\
(12, 2413, 132) \;$\leadsto$ & 2357614 \\
(132, 12, 2413) \;$\leadsto$ & 2457163 \\
(12, 132, 2413) \;$\leadsto$ & 2357164 \\
(2413, 213, 12) \;$\leadsto$ & 5246713 \\
(2413, 12, 213) \;$\leadsto$ & 6245713 \\
(213, 2413, 12) \;$\leadsto$ & 3246715 \\
(12, 2413, 213) \;$\leadsto$ & 6235714 \\
(213, 12, 2413) \;$\leadsto$ & 3245716 \\
(12, 213, 2413) \;$\leadsto$ & 4235716
\end{tabular}
\caption{The 30 elements of $\Av_{7,3}^{2\prec 7}(1324)$ not ending with $1$.} 
\label{tab:(7,3)}
\end{table}

%%%%%%%%%%%%%%%%%%%%%%%%%%%%%%%%%%%%%%%%%%%%
\section{Conjecture and final remarks}
\label{sec:final}

At the beginning of the paper, we introduced the notation
\[ T_{a, k}(x)=\sum\limits_{n=k+1}^{\infty}\abs{\Av_{n,k}^{a\prec n}(1324)}x^n \;\text{ and }\; g_a(x,t) = \sum\limits_{k=1}^{\infty}t^kT_{a,k}(x). \] 

Let $T_{1,0}(x)=x$ and $T_{a,0}(x)=\abs{\Av_{a-1}(1324)}x^a$. These functions account for the permutations of size $a$ that start with $a$. 

Recall that, in one-line notation, a permutation $\sigma\in \Av_{n,k}^{a\prec n}(1324)$ has all of its entries less than $a$ to the right of $n$, and if they are removed, we are left with a reduced $1324$-avoiding permutation of size $n-a+1$ having the 1 to the left of the maximal element. Thus, it is not unreasonable to expect a connection between $T_{a,k}(x)$ and $T_{1,k}(x)$. With that in mind, and based on how we proved Theorem~\ref{thm:T_2k}, we spent some time looking for an expansion of $T_{a,k}(x)$ in terms of powers of $f(x)$ and the functions $T_{a,j}(x)$ for $j=1,\dots,a-1$. Our search lead to the following conjecture.

\begin{conjecture} For $k\ge a$, we have the following equivalent formulas:
\begin{enumerate}[\rm (i)]
\itemindent10pt
\item $\sum\limits_{j=0}^k (-1)^j\binom{k}{j} f(x)^j T_{a,k-j} = 0$.
\item\label{itm:conjecture} 
$T_{a,k}(x)=\sum\limits_{j=0}^{a-1}\binom{k}{j}f(x)^{k-j}\sum\limits_{i=0}^{j} (-1)^{i}\binom{j}{i} f(x)^{i}T_{a,j-i}(x)$.
\item $T_{a,k}(x)=\sum\limits_{j=0}^{a-1} (-1)^{a-j-1}\binom{k}{j}\binom{k-j-1}{a-j-1} f(x)^{k-j}T_{a,j}(x)$.
\end{enumerate}
\end{conjecture}

For $a=1$ the statements are trivial, and for $a=2$, \eqref{itm:conjecture} becomes
\[ T_{2,k}(x) =  f(x)^{k}T_{2,0}(x) + k f(x)^{k-1}\big(T_{2,1}(x) - f(x)T_{2,0}(x)\big), \]
as claimed and proved in Theorem~\ref{thm:T_2k}. 

For $a=3$ the conjectured formula \eqref{itm:conjecture} becomes
\begin{align*}
T_{3,k}(x) &= f(x)^{k}T_{3,0}(x) \\
&\quad+ k f(x)^{k-1}\big(T_{3,1}(x) - f(x)T_{3,0}(x)\big) \\
&\quad+ \tbinom{k}{2} f(x)^{k-2} \big(T_{3,2}(x) - 2f(x)T_{3,1}(x) + f(x)^2T_{3,0}(x)\big),
\end{align*}
where $T_{3,0}(x)=2x^3$ (counting the permutations $312$ and $321$). 
Similar to what we did for $a=2$, we can interpret the term $f(x)^{k}T_{3,0}(x)$ as counting the permutations in $\Av_{n,k}^{3\prec n}(1324)$ that end with $12$ or $21$. Moreover, the function $k f(x)^{k-1}\big(T_{3,1}(x) - f(x)T_{3,0}(x)\big)$ can be interpreted as counting $k$-tuples $(\sigma_1,\dots,\sigma_k)$, where $\sigma_j\in \Av_{m_j,1}^{1\prec m_j}(1324)$, $m_j\ge 2$, and one of these permutations, say $\sigma_\ell$, is marked in such a way that it corresponds to an element of $\Av_{m_\ell,1}^{3\prec m_\ell}(1324)$ that does not end with $12$ or $21$.

\smallskip
Finally, rewriting the third component of $T_{3,k}(x)$ as
\[ \tbinom{k}{2} f(x)^{k-2} \big(T_{3,2}(x) - 2f(x)(T_{3,1}(x) - f(x)T_{3,0}(x)) - f(x)^2T_{3,0}(x)\big), \]
it can be argued that this function counts $k$-tuples $(\sigma_1,\dots,\sigma_k)$ of primitives, where two of them, say $\sigma_{\ell_1}$ and $\sigma_{\ell_2}$, are marked in such a way that the pair $(\sigma_{\ell_1},\sigma_{\ell_2})$ corresponds to a permutation in $\Av_{m,2}^{3\prec m}(1324)$, $m=m_{\ell_1}+m_{\ell_2}+1$, that does not end with $12$ or $21$.

\subsection*{Final remarks}
In this paper, we have introduced a notion of positional statistics that seems particularly suited to and provides a new way to think about 1324-avoiding permutations. As we did for $a=1$ in Theorem~\ref{thm:1<n} and for $a=2$ in Theorem~\ref{thm:2<n}, the ultimate goal is to find an expression for $g_a(x,t)$ in terms of known functions. 

Proving the above conjecture would be a significant step in that direction, but it is not the whole story. For instance, while the conjecture is true for $a=3$, we still need to find $T_{3,1}(x)$ and $T_{3,2}(x)$ in order to have a full expression for $g_3(x,t)$. For an arbitrary $a>3$, our conjecture would reduce the problem to finding $T_{a,j}$ for $j=1,\dots,a-1$.

We have observed interesting properties for several patterns and hope that our work motivates the community to explore  positional statistics for patterns other than $1324$. 

\acknowledgements
We would like to thank the referees for their thorough and thoughtful reports which helped us improve the quality of the paper.

%\bibliographystyle{abbrvnat}
%%%%%%%%%%%%%%%%%%%%%%%%%%%%%%%%%%%%%%%%%%%%

\end{document}